\documentclass{article}

\usepackage{amsmath,amssymb,amsthm,mathtools}

\newcommand{\argmin}{\operatornamewithlimits{arg\,min}}

\theoremstyle{plain}
\newtheorem{thm}{Theorem}
\newtheorem{lem}{Lemma}
\theoremstyle{definition}

\newtheorem{rmk}{Remark}
\newtheorem{ass}{Assumption}

\author{
	Clemens Kirisits\footnote{Johann Radon Institute for Computational and Applied Mathematics (RICAM),
Austrian Academy of Sciences, Linz, Austria},
	Otmar Scherzer\footnote{Computational Science Center, University of Vienna, and Johann Radon Institute for Computational and Applied Mathematics (RICAM), Austrian Academy of Sciences, Linz, Austria}}
\date{\today}
\title{A range condition for polyconvex variational regularization}

\begin{document}
\maketitle

\begin{abstract}
In the context of \emph{convex} variational regularization it is a known result that, under suitable differentiability assumptions, source conditions in the form of variational inequalities imply range conditions, while the converse implication only holds under an additional restriction on the operator. In this article we prove the analogous result for \emph{polyconvex} regularization. More precisely, we show that the variational inequality derived in \cite{KirSch17} implies that the derivative of the regularization functional must lie in the range of the dual-adjoint of the derivative of the operator. In addition, we show how to adapt the restriction on the operator in order to obtain the converse implication.
\end{abstract}

\section{Introduction}\label{sec:intro}
Consider a nonlinear operator equation with inexact data
$$ K(u) = v^\delta, \quad \|v^\delta - v^\dagger\| \le \delta,$$
where $K:U\to V$ acts between Banach spaces, $v^\dagger,v^\delta \in V$ are exact and noisy data, respectively, and $\delta>0$ is the noise level. A common method for the stable inversion of $K$ is variational regularization which consists in computing regularized solutions $u_\alpha^\delta$ as minimizers of functionals of the form
	\begin{equation}\label{eq:T}
		u \mapsto \mathcal{T}_\alpha (u;v^\delta) = \|K(u) - v^\delta\|^q + \alpha\mathcal{R}(u).
	\end{equation}
Here $\mathcal{R}$ is a typically convex regularization functional, $\alpha>0$ and $q\ge 1.$ A natural requirement for such methods is that regularized solutions converge, in some sense, to an exact solution as the noise level tends to zero. Convergence rates additionally provide bounds on the discrepancy between regularized and exact solutions in terms of the noise level. In a Banach space setting the most common measure of discrepancy is the Bregman distance associated to $\mathcal{R}$ \cite{BurOsh04}.

In order to guarantee convergence rates one has to impose a source condition of some sort. Traditionally, in a linear Hilbert space setting with quadratic Tikhonov regularization, this was done by assuming that the minimum norm solution lies in the range of an operator closely related to the adjoint of $K$. See \cite[Ch.~5]{EngHanNeu96} for example. Generalizing this range condition to the nonlinear Banach space setting outlined in the previous paragraph yields
	\begin{equation}\label{eq:inclusion}
		\mathcal{R}'(u^\dagger) \in \mathrm{ran}\,K'(u^\dagger)^\#,
	\end{equation}
where $u^\dagger$ is an $\mathcal{R}$-minimizing solution and $K'(u^\dagger)^\#$ is the dual-adjoint of the G\^ateaux derivative of $K$ at $u^\dagger$.

More recently, it was shown in \cite{HofKalPoeSch07} that convergence rates can also be obtained by assuming that a variational inequality like
	\begin{equation}\label{eq:inequality}
		\langle u^*, u^\dagger - u \rangle \le \beta_1 D_{u^*}(u;u^\dagger) + \beta_2 \| K(u) - v^\dagger\|
	\end{equation}
holds for all $u$ in a certain neighbourhood of $u^\dagger$. Here $u^*$ is a subgradient of $\mathcal{R}$ at $u^\dagger$ and $D_{u^*}(u;u^\dagger)$ denotes the corresponding Bregman distance between $u$ and $u^\dagger$. Note that \eqref{eq:inequality} does not require $K$ or $\mathcal{R}$ to be differentiable. If they are, however, then the variational inequality \eqref{eq:inequality} implies the range condition \eqref{eq:inclusion}. The converse implication only holds under an additional assumption on the nonlinearity of the operator $K$. For a more detailed discussion on the relations between the various types of source conditions we refer to \cite[pp.~70--73]{SchGraGroHalLen09}.

For certain inverse problems on $W^{1,p}(\Omega,\mathbb{R}^N)$, such as image or shape registration models inspired by nonlinear elasticity \cite{BurModRut13,IglRumSch17}, convex regularization is too restrictive, while the weaker notion of polyconvexity is more appropriate. Indeed, nonconvex regularization functionals $\mathcal{R}$ with polyconvex integrands are well-suited for deriving stable and convergent regularization schemes. However, since such functionals are not subdifferentiable in general, the question is how to obtain convergence rates. In \cite{KirSch17} we addressed this issue by following Grasmair's approach of generalized Bregman distances \cite{Gra10b}. First, we introduced the weaker concept of $W_{\mathrm{poly}}$-subdifferentiability, specifically designed for functionals with polyconvex integrands, and gave conditions for existence of $W_{\mathrm{poly}}$-subgradients. By means of the corresponding $W_{\mathrm{poly}}$-Bregman distance we were then able to translate the convergence rates result of \cite{HofKalPoeSch07} to the polyconvex setting. The source condition derived in \cite{KirSch17} reads
	\begin{equation}\label{eq:newinequality}
		w(u^\dagger) - w(u) \le \beta_1 D^{\mathrm{poly}}_w(u;u^\dagger) + \beta_2 \| K(u) - v^\dagger\|,
	\end{equation}
where $w$ is a $W_\mathrm{poly}$-subgradient of $\mathcal{R}$ at $u^\dagger$ and $D^{\mathrm{poly}}_w(u;u^\dagger)$ is the corresponding generalized Bregman distance.

The main results of the present paper are Theorems \ref{thm:rancon} and \ref{thm:converse} in Section \ref{sec:range}. Theorem \ref{thm:rancon} states that the variational inequality \eqref{eq:newinequality} implies the range condition \eqref{eq:inclusion}, given that $K$ and $\mathcal{R}$ are differentiable and $\mathcal{R}$ satisfies the conditions guaranteeing existence of a $W_\mathrm{poly}$-subgradient. A major part of the proof consists in showing that $\mathcal{R}'(u^\dagger) = w'(u^\dagger)$ in this case. Conversely, Theorem \ref{thm:converse} states that
$$w'(u^\dagger) \in \mathrm{ran}\,K'(u^\dagger)^\#$$
implies \eqref{eq:newinequality}, if the nonlinearities of $K$ and $w$ satisfy a certain inequality around $u^\dagger$. 

\section{Polyconvex functions and generalized Bregman distances} \label{sec:prelim}
This section is a brief summary of the most important prerequisites from \cite{KirSch17}. For $N,n \in \mathbb{N}$ we will frequently identify matrices in $\mathbb{R}^{N\times n}$ with vectors in $\mathbb{R}^{Nn}$.

\paragraph{Polyconvex functions.}  A function $f:\mathbb{R}^{N\times n} \to \mathbb{R} \cup \{+\infty\}$ is \emph{polyconvex}, if, for every $A \in \mathbb{R}^{N\times n}$, $f(A)$ can be written as a convex function of the minors of $A$. More precisely, let $1\le s \le \min(N,n) \eqqcolon N \wedge n$ and define $\sigma (s) \coloneqq \big( \begin{smallmatrix} n \\ s \end{smallmatrix} \big) \big( \begin{smallmatrix} N \\ s \end{smallmatrix} \big)$ as well as $\tau \coloneqq \sum_{s=1}^{N \wedge n} \sigma(s)$. Denote by $\mathrm{adj}_s A \in \mathbb{R}^{\sigma(s)}$ the matrix of all $s\times s$ minors of $A$ and set
	$$ T(A) \coloneqq (A,\mathrm{adj}_2 A,\ldots,\mathrm{adj}_{N\wedge n} A) \in \mathbb{R}^\tau.$$
Now, a function $f:\mathbb{R}^{N\times n} \to \mathbb{R} \cup \{+\infty\}$ is polyconvex, if there is a convex function $F:\mathbb{R}^\tau \to \mathbb{R} \cup \{+\infty\}$ such that $f = F \circ T$. Every convex function is polyconvex. The converse statement only holds, if $N \wedge n = 1.$ The importance of polyconvex functions in the calculus of variations is due to the fact that they render functionals of the form
$$\mathcal{R}(u) = \int_\Omega f( \nabla u(x))\, dx$$
weakly lower semicontinuous in $W^{1,p} (\Omega, \mathbb{R}^N)$, where $\Omega \subset \mathbb{R}^n$. For more details on polyconvex functions see \cite{Bal77, Dac08}.

\paragraph{The set $W_{\mathrm{poly}}$.} For the remainder of this article, unless stated otherwise, we let $\Omega \subset \mathbb{R}^n$ be an open set, $p\ge N \wedge n$, and set $U=W^{1,p}(\Omega,\mathbb{R}^N)$.

The following variant of the map $T$ will prove useful. Set $\tau_2 \coloneqq \sum_{s=2}^{N \wedge n} \sigma(s)$ and define
	$$ T_2(A) \coloneqq (\mathrm{adj}_2 A,\ldots,\mathrm{adj}_{N\wedge n} A) \in \mathbb{R}^{\tau_2}.$$
If $u \in U$, then $\mathrm{adj}_s \nabla u$ consists of sums of products of $s$ $L^p(\Omega)$ functions, and therefore, by H\"older's inequality, $\mathrm{adj}_s \nabla u \in L^{p/s}(\Omega,\mathbb{R}^{\sigma(s)})$. This motivates the following two defintions
$$ S \coloneqq \prod_{s=1}^{N\wedge n} L^{\frac{p}{s}} (\Omega, \mathbb{R}^{\sigma(s)}),\qquad S_2 \coloneqq \prod_{s=2}^{N\wedge n} L^{\frac{p}{s}} (\Omega, \mathbb{R}^{\sigma(s)}).$$
We define $W_{\mathrm{poly}}$ to be the set of all functions $w : U \to \mathbb{R}$ for which there is a pair $(u^*, v^*) \in U^* \times S_2^*$ such that
\begin{equation}\label{eq:waction}
	w(u) = \langle u^*,u\rangle_{U^*,U} + \langle v^*, T_2(\nabla u) \rangle_{S_2^*,S_2}
\end{equation}
for all $u\in U$. Note that, if $v^* = 0$, then $w$ can be identified with $u^* \in U^*$. Thus, the dual $U^*$ can be regarded a subset of $W_{\mathrm{poly}}$ in a natural way.

\paragraph{Generalized subgradients.}
Let $\mathcal{R}:U \to \mathbb{R} \cup \{+\infty\}$. We denote the effective domain of $\mathcal{R}$ by $\mathrm{dom}\,\mathcal{R} = \{u\in U: \mathcal{R}(u) < +\infty \}$. Following \cite{Gra10b,KirSch17,Sin97} we define the $W_{\mathrm{poly}}$\emph{-subdifferential} of $\mathcal{R}$ at $u \in \mathrm{dom}\,\mathcal{R}$ as 
	$$\partial_{\mathrm{poly}}\mathcal{R}(u) =  \{w \in W_{\!\mathrm{poly}} : \mathcal{R}(v) \ge \mathcal{R}(u) + w(v) - w(u) \text{ for all } v \in U\},$$
If $\mathcal{R}(u) = +\infty$, we set $\partial_{\mathrm{poly}}\mathcal{R}(u) = \emptyset$. The identification of $U^*$ with a subset of $W_{\mathrm{poly}}$ mentioned in the previous paragraph implies that $\partial \mathcal{R}(u) \subset \partial_{\mathrm{poly}}\mathcal{R}(u)$, that is, the classical subdifferential can be regarded a subset of the $W_{\mathrm{poly}}$-subdifferential. Elements of $\partial_{\mathrm{poly}}\mathcal{R}(u)$ are called $W_{\mathrm{poly}}$\emph{-subgradients} of $\mathcal{R}$ at $u$. Concerning existence of $W_{\mathrm{poly}}$-subgradients we have shown the following result in \cite{KirSch17}.
\begin{lem} \label{thm:wsubdif}
Let
	$$F:\Omega \times \mathbb{R}^N \times \mathbb{R}^{\tau} \to \mathbb{R}_{\ge 0} \cup \{+ \infty\}$$
be a Carath\'eodory function. Assume that, for almost every $x\in \Omega$, the map $(u,\xi) \mapsto F(x,u,\xi)$ is convex and differentiable throughout its effective domain and denote its derivative by $F'_{u,\xi}$. Let $p\in[1,\infty)$ and define the following functional on $U = W^{1,p}(\Omega,\mathbb{R}^N)$
	$$ \mathcal{R}(u) = \int_\Omega F(x,u(x),T(\nabla u(x))) \, dx.$$
If $\mathcal{R}(\bar{v}) \in \mathbb{R}$ and the function $x \mapsto F'_{u,\xi} (x,\bar{v}(x),T(\nabla \bar{v}(x)))$ lies in $L^{p^*}(\Omega,\mathbb{R}^N) \times S^*$, where $p^*$ denotes the H\"older conjugate of $p$, then this function is a $W_{\mathrm{poly}}$-subgradient of $\mathcal{R}$ at $\bar v$.
\end{lem}
\begin{rmk}\label{rmk:subgrad}
If $F'_{u,\xi} (\cdot,\bar{v}(\cdot),T(\nabla \bar{v}(\cdot)))$ is a $W_{\mathrm{poly}}$-subgradient $w \in \partial_\mathrm{poly} \mathcal{R}(\bar v) \subset W_\mathrm{poly}$, as postulated by Lemma \ref{thm:wsubdif}, then it must be possible to write its action on $u\in U$ in terms of a pair $(u^*, v^*) \in U^* \times S_2^*$ as in \eqref{eq:waction}. In order to do so recall that $T(A) = (A,T_2(A))$. We can split the variable $\xi \in \mathbb{R}^\tau$ accordingly into $(\xi_1, \xi_2) \in \mathbb{R}^{Nn} \times \mathbb{R}^{\tau_2}.$ Similarly, we can write $F'_{u,\xi} = (F'_u, F'_{\xi}) = (F'_u, F'_{\xi_1}, F'_{\xi_2}).$ Now we have
\begin{align*}
	w(u)	&= \int_{\Omega} F'_{u,\xi} (x,\bar{v}(x),T(\nabla \bar{v}(x))) \cdot (u,T(\nabla u)) \, dx \\
			&= \int_{\Omega} F'_{u} (x,\bar{v}(x),T(\nabla \bar{v}(x))) \cdot u(x) \, dx \\
			&\quad {}+ \int_{\Omega} F'_{\xi_1} (x,\bar{v}(x),T(\nabla \bar{v}(x))) \cdot \nabla u(x) \, dx \\
			&\quad {}+ \int_{\Omega} F'_{\xi_2} (x,\bar{v}(x),T(\nabla \bar{v}(x))) \cdot T_2(\nabla u(x)) \, dx.
\end{align*}
The integral in the bottom line corresponds to the dual pairing $\langle v^*, T_2(\nabla u) \rangle_{S_2^*,S_2}$ in \eqref{eq:waction}, while the previous two terms correspond to $\langle u^*,u\rangle_{U^*,U}$. Therefore, $u^*$ is given by $(F'_u, F'_{\xi_1})$ and $v^*$ by $F'_{\xi_2}$. Also note that all integrals are well-defined and finite because of the integrability conditions on the derivative of $F$ in Lemma \ref{thm:wsubdif}.
\end{rmk}

\paragraph{Generalized Bregman distances.}
Whenever $\mathcal{R}$ has a $W_{\mathrm{poly}}$-subgradient $w \in \partial_{\mathrm{poly}}\mathcal{R}(u)$ we can define the associated $W_{\mathrm{poly}}$\emph{-Bregman distance} between $v\in U$ and $u$ as
	$$ D^{\mathrm{poly}}_w (v;u) = \mathcal{R}(v) - \mathcal{R}(u) - w(v) + w(u).$$
Note that, just like the classical Bregman distance, the $W_{\mathrm{poly}}$-Bregman distance is nonnegative, satisfies $ D^{\mathrm{poly}}_w (u;u) = 0$ whenever defined, and is \emph{not} symmetric with respect to $u$ and $v$. In addition, if $w = (u^*,0) \in \mathcal{R}_\mathrm{poly}(u)$, then $u^* \in \partial\mathcal{R}(u)$ and the classical and $W_\mathrm{poly}$-Bregman distances coincide, that is,
$$D^{\mathrm{poly}}_w (v;u) = D_{u^*} (v;u).$$
See \cite{Gra10b,SchGraGroHalLen09} for more details on (generalized) Bregman distances.

In order to be able to quote the source condition from \cite{KirSch17} we need one more definition: We call $u^\dagger \in U$ an $\mathcal{R}$\emph{-minimizing solution}, if it solves the exact operator equation and minimizes $\mathcal{R}$ among all other exact solutions, that is,
$$ u^\dagger \in \argmin \big\{\mathcal{R}(u) : u\in U, K(u) = v^\dagger\big\}.$$
\begin{ass}\label{ass:wsourcecon}
Assume that $\mathcal{R}$ has a $W_{\mathrm{poly}}$-subgradient $w$ at an $\mathcal{R}$-minimizing solution $u^\dagger$ and that there are constants $\beta_1 \in [0,1)$, $\beta_2, \bar\alpha > 0$ and $\rho > \bar\alpha\mathcal{R}(u^\dagger)$ such that
	\begin{equation}\label{eq:wsourcecon}
		w(u^\dagger) - w(u) \le \beta_1 D^{\mathrm{poly}}_w(u;u^\dagger) + \beta_2 \| K(u) - v^\dagger\|
	\end{equation}
	holds for all $u$ with $\mathcal{T}_{\bar \alpha} (u;v^\dagger) \le \rho$.
\end{ass}

\section{A range condition}\label{sec:range}

At the end of this section we prove our main results, Theorems \ref{thm:rancon} and \ref{thm:converse}. Before that we have to state a few preliminary results. First, we recall the definition of the dual-adjoint operator together with a characterization of its range (Lemma \ref{thm:dualadjointrange}). Next, we compute the G\^ateaux derivative of
$$ \mathcal{R}(u) = \int_\Omega f(x,u(x),\nabla u(x)) \, dx$$
in Lemma \ref{thm:Rprime}, and of $w\in W_\mathrm{poly}$ in Lemma \ref{thm:wprime}, respectively.

For every bounded linear operator $A : U \to V$ acting between locally convex spaces there exists a unique operator $A^\#:V^* \to U^*$, also linear and bounded, satisfying
	$$ \langle A^\#v^*,u \rangle_{U^*,U} = \langle v^*, Au \rangle_{V^*,V}$$
for all $u\in U$ and $v^* \in V^*.$ See, for instance, Section VII.1 of \cite{Yos65}. The operator $A^\#$ is called the \emph{dual-adjoint} of $A$.
\begin{lem}\label{thm:dualadjointrange}
Let $U,V$ be normed linear spaces, $A:U\to V$ a bounded linear operator and $u^*\in U^*$. Then $u^* \in \mathrm{ran}\,A^\#$, if and only if there is a $C>0$ such that
	$$ |\langle u^*,u \rangle_{U^*,U}| \le C \|Au\|$$
for all $u \in U$.
\end{lem}
\begin{proof}
See Lemma 8.21 in \cite{SchGraGroHalLen09}.
\end{proof}

Let $ K: \mathcal{D}(K) \subset U \to V $ be a map acting between normed spaces and let $u \in \mathcal{D}(K)$, $h \in U.$ If the limit
	$$ K'(u;h) = \lim_{t\to 0^+} \frac{1}{t} \big( K(u+th) - K(u)\big) $$
exists in $V$, then $K'(u;h)$ is called the \emph{directional derivative} of $K$ at $u$ in direction $h$. If $K'(u;h)$ exists for all $h \in U$ and there is a bounded linear operator $K'(u):U \to V $ satisfying
	$$ K'(u)h = K'(u;h)$$
for all $h \in U$, then $K$ is \emph{G\^ateaux differentiable} at $u$ and $K'(u)$ is called the \emph{G\^ateaux derivative} of $K$ at $u$.

\begin{lem}\label{thm:Rprime}
Let
	$$ f:\Omega \times \mathbb{R}^N \times \mathbb{R}^{N\times n} \to \mathbb{R}_{\ge 0} \cup \{+\infty\}$$
be a nonnegative Carath\'eodory function. Assume that, for almost every $x\in \Omega$, the map $(u,A) \mapsto f(x,u,A)$ is differentiable throughout its effective domain and that
\begin{equation}\label{eq:factcond}
	|f'_{u,A}(x,u,A)| \le a(x) + b |u|^{p-1} + c |A|^{p-1}
\end{equation}
holds there for $p\ge 1$ and some $a\in L^{p^*}(\Omega)$ and $b,c\ge 0$. Then the functional
$$\mathcal{R}:U = W^{1,p}(\Omega, \mathbb{R}^N) \to \mathbb{R}_{\ge 0} \cup \{+\infty\},$$
defined by
$$ \mathcal{R}(u) = \int_\Omega f(x,u(x),\nabla u(x)) \, dx,$$
is G\^ateaux differentiable in the interior of its effective domain. Its G\^ateaux derivative at $u\in\mathrm{int\,dom}\,\mathcal{R}$ is given by
\begin{equation}\label{eq:Rprime}
\begin{aligned}
	\langle \mathcal{R}'(u), \hat u \rangle_{U^*,U}
		&= \int_\Omega f'_u(x,u(x),\nabla u(x)) \cdot \hat u(x)\,dx  \\
		&\quad  {}+ \int_\Omega f'_A(x,u(x),\nabla u(x)) \cdot \nabla \hat u(x) \,dx, \qquad \hat u \in U.
\end{aligned}
\end{equation}
\end{lem}
\begin{proof}
Fix $u\in\mathrm{int\,dom}\,\mathcal{R}$ and $\hat u \in U$. Assuming we can differentiate under the integral sign we have
\begin{align*}
	\mathcal{R}'(u;\hat u)
		&= \lim_{t\to 0^+} \frac{1}{t} \big( \mathcal{R}(u+t\hat u) - \mathcal{R}(u)\big) \\
		&= \int_{\Omega} \lim_{t\to 0^+} \frac{1}{t} \big( f(x,u+t\hat u, \nabla u + t \nabla \hat u) - f(x,u,\nabla u)\big) \, dx \\
		&= \int_{\Omega} \partial_t f(x,u+t\hat u, \nabla u + t \nabla \hat u)\Big\vert_{t=0} \, dx \\
		&= \int_\Omega \left(f'_u(x,u,\nabla u) \cdot \hat u + f'_A(x,u,\nabla u) \cdot \nabla \hat u \right)\,dx,
\end{align*}
which is just \eqref{eq:Rprime}.

It remains to show that differentiation and integration are interchangeable. For $\epsilon >0$ sufficiently small (see below) we define $g:(-\epsilon,\epsilon)\times \Omega \to \mathbb{R}_{\ge 0} \cup \{+\infty\}$,
	$$g(t,x) = f(x,u(x)+t\hat u(x), \nabla u(x) + t \nabla \hat u(x)).$$
The identity $\partial_t \int_\Omega g(t,x)\, dx = \int_\Omega \partial_t g(t,x)\, dx$ holds true, if the following three conditions are satisfied.
\begin{enumerate}
	\item Integrability: The function $x\mapsto g(t,x)$ is integrable for all $t \in (-\epsilon,\epsilon)$.
	\item Differentiability: The partial derivative $\partial_t g(t,x)$ exists for almost every $x\in \Omega $ and all $t \in (-\epsilon,\epsilon)$.
	\item Uniform upper bound: There is a function $h\in L^1(\Omega)$ such that $|\partial_t g(t,x)| \le h(x)$ for almost every $x\in\Omega$ and all $t \in (-\epsilon,\epsilon)$.
\end{enumerate}
Item 1 is satisfied, since $u$ lies in the interior of $\mathrm{dom}\,\mathcal{R} $ and therefore
	$$ \int_\Omega |g(t,x)| \, dx = \mathcal{R}(u+t \hat u) < \infty, \qquad -\epsilon < t < \epsilon, $$
for $\epsilon$ sufficiently small. In particular, $g(t,x) < \infty$ for almost every $x$ and every $t\in (-\epsilon ,\epsilon)$.
Thus, item 2 holds as well. Concerning item 3, we use assumption \eqref{eq:factcond} to obtain for almost every $x \in \Omega$
\begin{align*}
	|\partial_t g(t,x)|
		&= |f'_u(x,u+t\hat u, \nabla u + t\nabla \hat u)\cdot \hat u + f'_A(x,u+t\hat u, \nabla u + t\nabla \hat u)\cdot \nabla \hat u| \\
		&\le |f'_u(x,u+t\hat u, \nabla u + t\nabla \hat u)| |\hat u| + |f'_A(x,u+t\hat u, \nabla u + t\nabla \hat u)| |\nabla \hat u| \\
		&\le (|\hat u| + |\nabla \hat u|)(a + b |u + t \hat u|^{p-1} + c |\nabla u + t \nabla \hat u|^{p-1}).
\end{align*}
We estimate further
$$ |u + t \hat u|^{p-1} \le \left( |u| + |t| |\hat u|\right)^{p-1} \le \max\{1,2^{p-2}\} \left( |u|^{p-1} + \epsilon^{p-1} |\hat u|^{p-1}\right)$$
and similarly
$$ |\nabla u + t \nabla \hat u|^{p-1} \le \max\{1,2^{p-2}\} \left( |\nabla u|^{p-1} + \epsilon^{p-1} |\nabla \hat u|^{p-1}\right).$$
Thus we have found an upper bound for $|\partial_t g(t,x)|$, which is independent of $t$. This bound is essentially a sum of products of the form $y(x)z(x)^{p-1}$, where $y,z \in L^p(\Omega)$. Since, in this case, $z^{p-1}$ lies in $L^{p^*}(\Omega)$, H\"older's inequality shows that $yz^{p-1} \in L^1(\Omega).$
\end{proof}

\begin{lem} \label{thm:wprime}
The functions $w\in W_{\mathrm{poly}}$ are G\^ateaux differentiable on all of $U$. Identifying $w$ with $(u^*,v^*) \in U^* \times S_2^*$ its G\^ateaux derivative at $u\in U$ is given by
$$\langle w'(u), \hat u \rangle_{U^*,U} = \langle u^*, \hat u \rangle_{U^*,U} + \int_{\Omega} v^*(x) \cdot T'_2(\nabla u(x)) \nabla \hat u(x) \, dx, \qquad \hat u \in U,$$
where $T'_2(\nabla u(x))$ denotes the derivative of the map $T_2:\mathbb{R}^{Nn} \to \mathbb{R}^{\tau_2}$ at $\nabla u(x)$.
\end{lem}
\begin{proof}
Identify $w\in W_{\!\mathrm{poly}}$ with $(u^*,v^*) \in U^* \times S_2^*$ and let $u,\hat u \in U.$ First, we separate the linear and nonlinear parts of $w$.
\begin{align*}
	w'(u;\hat u)
		&= \lim_{t\to 0^+} \frac{1}{t} \big( w(u+t\hat u) - w(u)\big) \\
		&= \lim_{t\to 0^+} \frac{1}{t} \big( \langle u^*,u+t\hat u \rangle_{U^*,U} + \langle v^*, T_2(\nabla u +t \nabla \hat u) \rangle_{S_2^*,S_2} \\
		&\qquad {} - \langle u^*,u \rangle_{U^*,U} - \langle v^*, T_2(\nabla u) \rangle_{S_2^*,S_2} \big) \\
		&= \langle u^*,\hat u \rangle_{U^*,U} + \lim_{t\to 0^+} \frac{1}{t} \langle v^*, T_2(\nabla u + t \nabla \hat u) - T_2(\nabla u) \rangle_{S_2^*,S_2} 
\end{align*}
Assuming we can differentiate under the integral sign, the remaining limit equals
\begin{align*}
	&\quad \lim_{t\to 0^+} \frac{1}{t} \langle v^*, T_2(\nabla u + t \nabla \hat u) - T_2(\nabla u) \rangle_{S_2^*,S_2} \\
	&= \int_{\Omega} \lim_{t\to 0^+} \frac{1}{t} \Big[ v^* \cdot \big(T_2 (\nabla u + t \nabla \hat u) - T_2(\nabla u) \big) \Big] \, dx \\
	&= \int_{\Omega} \partial_t \Big[ v^* \cdot T_2 (\nabla u + t \nabla \hat u) \Big]_{t=0} \, dx \\
	&= \int_{\Omega} v^* \cdot T'_2(\nabla u)\nabla \hat u \, dx.
\end{align*}
As in the proof of Lemma \ref{thm:Rprime} we have to check the conditions for interchanging integration and differentiation. Define the function
$$ g(t,x) =  v^*(x) \cdot T_2 (\nabla u(x) + t \nabla \hat u(x))$$
on $(-\epsilon,\epsilon) \times \Omega$. It is integrable for all $t$, since $T_2$ maps $L^p(\Omega, \mathbb{R}^{N\times n})$ into $S_2$ and $v^*$ lies in $S_2^*$. It is also differentiable with respect to $t$, since the entries of $T_2(\nabla u(x) + t \nabla \hat u(x))$ are polynomials in $t$. Finally, $\partial_t g$ can be bounded in the following way
\begin{align}
	|\partial_t g|
		&= \Big|\partial_t \sum_{s=2}^n v_s^* \cdot \mathrm{adj}_s (\nabla u + t \nabla \hat u) \Big| \notag \\
		&= \Big|\sum_{s=2}^n v_s^* \cdot \mathrm{adj}'_s (\nabla u + t \nabla \hat u)\nabla \hat u \Big| \notag \\
		&\le \big| \nabla \hat u\big| \sum_{s=2}^n \big| v_s^* \big| \big| \mathrm{adj}'_s (\nabla u + t \nabla \hat u) \big| \label{eq:upperbounddtg}
\end{align}
where $v^*_s$ denotes the $L^{(\frac{p}{s})^*}(\Omega,\mathbb{R}^{\sigma(s)})$-component of $v^*.$ The derivative $\mathrm{adj}'_s (\nabla u + t \nabla \hat u)$ consists of sums of products of $s-1$ terms of the form $\partial_{x_i}u_j + t \partial_{x_i}\hat{u}_j$. After expanding, every such product can be bounded by
	\begin{equation}\label{eq:adjprimebound}
		\sum_{k=0}^{s-1} |t|^k \sum_m |g_{km}| \le \sum_{k=0}^{s-1} \epsilon^k \sum_m |g_{km}|,
	\end{equation}
where each $g_{km}$ is a product of $s-1$ $L^p$ functions and therefore lies in $L^{\frac{p}{s-1}}$. Combining \eqref{eq:upperbounddtg} with \eqref{eq:adjprimebound} gives an upper bound for $\partial_t g$ which is independent of $t$. Using H\"older's inequality it is now straightforward to verify that this bound is indeed an $L^1$ function.
\end{proof}

\begin{thm}\label{thm:rancon}
Let $\mathcal{R}$ satisfy the requirements of Lemma \ref{thm:wsubdif} at an $\mathcal{R}$-minimizing solution $u^\dagger\in \mathrm{int\,dom\,}\mathcal{R}$ and let $w$ be the $W_{\mathrm{poly}}$-subgradient thus provided. Suppose Assumption \ref{ass:wsourcecon} holds for this $u^\dagger$ and $w$. Moreover, assume that the integrand $f$ of $\mathcal{R}$ satisfies inequality \eqref{eq:factcond} and that $K$ is G\^ateaux differentiable at $u^\dagger$. Then $\mathcal{R}$ is G\^ateaux differentiable at $u^\dagger$ and
	$$\mathcal{R}'(u^\dagger) = w'(u^\dagger) \in \mathrm{ran}\, K'(u^\dagger)^\#.$$
\end{thm}
\begin{proof}
The proof consists of two steps. First, we show that the source condition implies that
\begin{equation}\label{eq:sourcecondif}
0 \le \beta_1 \langle \mathcal{R}'(u^\dagger), \hat u \rangle_{U^*,U} + (1-\beta_1) \langle w'(u^\dagger), \hat u \rangle_{U^*,U} + \beta_2 \|K'(u^\dagger)\hat u\|
\end{equation}
holds for all $\hat u \in U$. Second, the derivatives of $\mathcal{R}$ and $w$ at $u^\dagger$ agree, which leads to
$$ \langle \mathcal{R}'(u^\dagger), \hat u \rangle_{U^*,U} \le \beta_2 \|K'(u^\dagger)\hat u\|$$
for all $\hat u \in U$. Finally, Lemma \ref{thm:dualadjointrange} implies $\mathcal{R}'(u^\dagger) \in \mathrm{ran}\, K'(u^\dagger)^\#$.

Step 1: Inequality \eqref{eq:wsourcecon} can be equivalently written as
$$ 0 \le \beta_1 (\mathcal{R}(u) - \mathcal{R}(u^\dagger)) + (1-\beta_1)(w(u) - w(u^\dagger)) + \beta_2 \|K(u) - K(u^\dagger)\|.$$
Since $\mathcal{R}$ satisfies the requirements of Lemma \ref{thm:wsubdif} as well as inequality \eqref{eq:factcond}, Lemma \ref{thm:Rprime} applies. Now, because of differentiability of both $\mathcal{R}$ and $K$ at $u^\dagger$ and because $\mathcal{T}_{\bar \alpha}(u^\dagger;v^\dagger) < \rho$ by Assumption \ref{ass:wsourcecon}, there is a $t_0>0$ for every $\hat u \in U$ such that $\mathcal{T}_{\bar \alpha}(u^\dagger+t\hat u;v^\dagger) < \rho$ for $0\le t< t_0$. Therefore,
$$ 0 \le \beta_1 (\mathcal{R}(u^\dagger+t\hat u) - \mathcal{R}(u^\dagger)) + (1-\beta_1)(w(u^\dagger+t\hat u) - w(u^\dagger)) + \beta_2 \|K(u^\dagger+t\hat u) - K(u^\dagger)\|.$$
Dividing by $t$ and letting $t\to 0$ yields \eqref{eq:sourcecondif}.

Step 2: We now show that $\mathcal{R}'(u^\dagger)=w'(u^\dagger)$. By Lemma \ref{thm:Rprime} the derivative of $\mathcal{R}$ is given by
\begin{align*}
	\langle \mathcal{R}'(u^\dagger), \hat u \rangle_{U^*,U}
		&= \int_\Omega f'_u(x,u^\dagger,\nabla u^\dagger) \cdot \hat u\,dx  + \int_\Omega  f'_A(x,u^\dagger,\nabla u^\dagger) \cdot \nabla \hat u \,dx. \\
		\intertext{Since $f(x,u,A) = F(x,u,T(A))$, the chain rule yields}
	\langle \mathcal{R}'(u^\dagger), \hat u \rangle_{U^*,U}
		&= \int_\Omega F'_u(x,u^\dagger,T(\nabla u^\dagger)) \cdot \hat u\,dx \\
		&\qquad {}+ \int_\Omega F'_{\xi}(x,u^\dagger,T(\nabla u^\dagger)) \cdot T' (\nabla u^\dagger) \nabla \hat u \,dx. \\
		\intertext{Now we split $F'_{\xi}$ into $(F'_{\xi_1},F'_{\xi_2})$ as in Remark \ref{rmk:subgrad} and, accordingly, $T' (\nabla u^\dagger)$ into $(\mathrm{Id},T_2'(\nabla u^\dagger))$ where $\mathrm{Id}$ is the identity mapping on $\mathbb{R}^{Nn}$. This leads to}
	\langle \mathcal{R}'(u^\dagger), \hat u \rangle_{U^*,U}
		&= \int_\Omega F'_u(x,u^\dagger,T(\nabla u^\dagger)) \cdot \hat u\,dx + \int_\Omega F'_{\xi_1}(x,u^\dagger,T(\nabla u^\dagger)) \cdot \nabla \hat u \,dx \\
		&\qquad {}+ \int_\Omega F'_{\xi_2}(x,u^\dagger,T(\nabla u^\dagger)) \cdot T'_2 (\nabla u^\dagger) \nabla \hat u \,dx.
\end{align*}
On the other hand, recall Remark \ref{rmk:subgrad} to see that the $W_\mathrm{poly}$-subgradient $w \in \partial_{\mathrm{poly}}\mathcal{R} (u^\dagger)$ provided by Lemma \ref{thm:wsubdif} is given by
\begin{align*}
	w(u)	&= \underbrace{\int_{\Omega} F'_{u} (x,u^\dagger),T(\nabla u^\dagger)) \cdot u \, dx + \int_{\Omega} F'_{\xi_1} (x,u^\dagger,T(\nabla u^\dagger)) \cdot \nabla u \, dx}_{= \langle u^*, u \rangle_{U^*,U}} \\
			&\quad {}+ \underbrace{\int_{\Omega} F'_{\xi_2} (x,u^\dagger,T(\nabla u^\dagger)) \cdot T_2(\nabla u) \, dx}_{=\langle v^*, T_2(\nabla u) \rangle_{S_2^*,S_2}}.
\end{align*}
Computing the derivative of $w$ according to Lemma \ref{thm:wprime} shows that $\mathcal{R}'(u^\dagger)=w'(u^\dagger)$.
\end{proof}

\begin{rmk}
Theorem \ref{thm:rancon} is an extension of its counterpart from convex regularization theory, Proposition 3.38 in \cite{SchGraGroHalLen09}, in the following sense. If the latter applies to a variational regularization method on $U$ with $\mathcal{R}$ being as in Lemma \ref{thm:Rprime} but convex, then Thm.~\ref{thm:rancon} applies as well with $w\in \partial_\mathrm{poly}\mathcal{R}(u^\dagger)$ and $ D_w^{\mathrm{poly}}(u;u^\dagger)$ reducing to their classical analogues and the respective variational inequalities and range conditions being identical. See also \cite[Remark 4.5]{KirSch17}.
\end{rmk}

\begin{thm}\label{thm:converse}
Assume $K$ is G\^ateaux differentiable at an $\mathcal{R}$-minimizing solution $u^\dagger$ and that $\mathcal{R}$ has a $W_{\mathrm{poly}}$-subgradient $w$ there. In addtion, suppose there is a $\omega^* \in V^*$ as well as constants $\beta_1\in [0,1)$, $\bar{\alpha} > 0$, $\rho > \bar{\alpha}\mathcal{R}(u^\dagger)$ such that
	\begin{equation}\label{eq:wprimerangeK}
		w'(u^\dagger) = K'(u^\dagger)^\# \omega^*,\quad \text{and}
	\end{equation}
	\begin{equation}\label{eq:Kwnonlinearity}
	\begin{aligned}
		\|\omega^*\| & \|K(u) - v^\dagger - K'(u^\dagger)(u-u^\dagger)\| + w(u^\dagger) - w(u) \\
		&\quad {}- \langle w'(u^\dagger),u^\dagger - u\rangle_{U^*,U} \le \beta_1 D_w^{\mathrm{poly}}(u;u^\dagger)
	\end{aligned}
	\end{equation}
for all $u$ satisfying $\mathcal{T}_{\bar{\alpha}}(u;v^\dagger) \le \rho.$ Then Assumption \ref{ass:wsourcecon} holds.
\end{thm}
\begin{proof}
The proof is along the lines of \cite[Prop.~3.35]{SchGraGroHalLen09}. We include it here in order to clarify the main differences.

By virtue of \eqref{eq:wprimerangeK} we have for every $u\in U$
	\begin{align*}
		\langle w'(u^\dagger), u^\dagger - u \rangle_{U^*,U}
			&= \langle K'(u^\dagger)^\# \omega^*, u^\dagger - u \rangle_{U^*,U} \\
			&= \langle \omega^*, K'(u^\dagger)(u^\dagger - u) \rangle_{U^*,U} \\
			&= \| \omega^* \| \| K'(u^\dagger)(u^\dagger - u)\| \\
			&\le \| \omega^* \| \| K(u) - v^\dagger\| + \| \omega^* \| \|K(u) - v^\dagger -  K'(u^\dagger)(u-u^\dagger)\|.
	\end{align*}
Adding $w(u^\dagger) - w(u) - \langle w'(u^\dagger),u^\dagger - u\rangle_{U^*,U}$ on both sides and using \eqref{eq:Kwnonlinearity} we arrive at
$$ w(u^\dagger) - w(u) \le \| \omega^* \| \| K(u) - v^\dagger\| + \beta_1 D_w^{\mathrm{poly}}(u;u^\dagger),$$
which is is just \eqref{eq:wsourcecon} with $\beta_2 = \| \omega^* \|$.
\end{proof}
\begin{rmk}
Note that the expression
\begin{equation}\label{eq:wnonlin}
	w(u^\dagger) - w(u) - \langle w'(u^\dagger),u^\dagger - u\rangle_{U^*,U}
\end{equation}
in \eqref{eq:Kwnonlinearity} is just the difference between $w$ and its continuous affine approximation around $u^\dagger.$ Therefore, condition \eqref{eq:Kwnonlinearity} is essentially a restriction on the nonlinearity of $K$ plus the nonlinearity of $w$, both computed in a neighbourhood of $u^\dagger.$

Theorem \ref{thm:converse} extends \cite[Prop.~3.35]{SchGraGroHalLen09} in the same way Theorem \ref{thm:rancon} extends \cite[Prop.~3.38]{SchGraGroHalLen09}. If $w = (u^*,0)$, then $w'(u^\dagger) = u^*$ and the nonlinearity \eqref{eq:wnonlin} vanishes.
\end{rmk}

\section{Conclusion}
In recent years, several authors have shown that nonconvex regularization of inverse problems is not only a viable possibility, but can even be preferable to convex regularization in certain situations, see for instance \cite{BreLor09,BurModRut13,Gra10b,IglRumSch17,KirSch17,Zar09}. However, convergence rates results for nonconvex regularization are exceedingly rare, let alone results relating different types of source conditions.

In this paper we have shown that two such results can be translated to the polyconvex setting of \cite{KirSch17}. The first one states that, under suitable differentiablity assumptions, source conditions in the form of variational inequalities imply range conditions. One of the reasons why this statement remains true is the fact that the derivative of $\mathcal{R}$ is equal to the derivative of its $W_\mathrm{poly}$-subgradient. This fact can be interpreted as a generalization of the well-known identity $\partial \mathcal{R}(u) = \{\mathcal{R}'(u)\}$ for convex and differentiable functions $\mathcal{R}$. Second, we have demonstrated that a converse statement can be obtained as well, given that the sum of the nonlinearities of $K$ and of the $W_\mathrm{poly}$-subgradient can be bounded by the $W_\mathrm{poly}$-Bregman distance around $u^\dagger$.

\subsection*{Acknowledgements}
Both authors acknowledge support by the Austrian Science Fund (FWF): S117. In addition, the work of OS is supported by the FWF Sonderforschungsbereich (SFB) F 68, as well as by project I 3661, jointly funded by FWF and Deutsche Forschungsgemeinschaft (DFG).

\def\cprime{$'$} \providecommand{\noopsort}[1]{}

\end{document}